\documentclass[11pt]{article}

\usepackage{amsmath, amsthm, amssymb, bbm, url} 
\usepackage[usenames]{color}\usepackage{graphicx}
\newcommand{\scr}{\mathcal}

\newcommand{\N}{\mathbb{N}}

\newcommand{\eps}{\varepsilon}

\renewcommand{\P}[1]{\mathbb{P}\left[ #1 \right]}
\newcommand{\E}[1]{\mathbb{E}\left[ #1 \right]}

\newcommand{\of}[1]{\left( #1 \right) }

\newcommand{\abs}[1]{\left| #1 \right|}
\newcommand{\angbs}[1]{\left< #1 \right>}

\newcommand{\sqbs}[1]{\left[ #1 \right]}

\newtheorem{maintheorem}{Theorem}

\newtheorem{maincoro}{Corollary}

\newtheorem*{conjecture*}{Conjecture}

\newtheorem*{theorem*}{Theorem}

\newcommand{\bfrac}[2]{\left(\frac{#1}{#2}\right)}

\newcommand{\ra}{\rightarrow}

\newcommand{\set}[1]{\left\{#1\right\}}

\def\seq{\subseteq}
\def\es{\emptyset}

\def\Ex{\mathbb{E}}

\def\Pr{\mathbb{P}}

  \def\d{\delta}

 \def\m{\mu}  
   
   \def\Om{\Omega}

\newcommand{\beqn}{\begin{equation}}
\newcommand{\eeqn}{\end{equation}}
\newcommand{\zon}{\set{0,1}^n}
\newcommand{\bfo}{\mathbbm{1}}
\numberwithin{equation}{section}
\newcommand{\fse}{f^{=S}}
\newcommand{\fes}{f^{=S}}
\newcommand{\fss}{f^{\subseteq S}}
\newcommand{\hf}{\widehat{f}}
\newcommand{\infl}{\textbf{Inf}}
\newcommand{\tinfl}{\mathbbm{I}}
\newcommand{\norm}[1]{\left\|#1\right\|}
\newcommand{\sumin}{\sum_{i=1}^n}
\newcommand{\sumtenc}{\sum_{0 < \abs{S} \le 10C}}

\title{On Sharp Thresholds of Monotone Properties: Bourgain's Proof Revisited.}

\author{Deepak Bal\thanks{dbal@cmu.edu}\\
Department of Mathematical Sciences,\\ Carnegie Mellon University,\\ Pittsburgh, PA 15213.
}
\date{}
\begin{document}
\maketitle



\begin{abstract}
The purpose of this expository note is to give the proof of a theorem of Bourgain with some additional details and updated notation. The theorem first appeared as an appendix to the breakthrough paper by Friedgut, \emph{Sharp Thresholds of graph properties and the $k$-SAT Problem} \cite{F99}.  Throughout, we use notation and definitions akin to those in O'Donnell's book, \emph{Analysis of Boolean Functions} \cite{O12}. 
\end{abstract}

\newcommand{\pprod}[1]{\pi^{\otimes #1}}
\newcommand{\Egen}[1]{\mathop{\Ex}_{x\sim \pprod{n}}\sqbs{#1}}

\section{Background and Theorem}
Random structures often exhibit what is called a threshold phenomenon. That is, a relatively small change in a parameter can cause a swift change in the structure of the overall system. In the random graph $G(n,p)$, the probability space consisting of $n$ vertices and edge probability $p$, this phenomenon is a central object of study. 
In his 1999 paper, \emph{Sharp Thresholds of graph properties and the $k$-SAT Problem}, Friedgut gave a simple characterization of monotone graph properties with coarse thresholds. 
The result is important because unlike results which preceded, it holds when $p=p(n)\rightarrow 0$ like $n^{-\Theta(1)}$ which is a range in which many thresholds occur.  In the appendix to that paper, Bourgain gave a characterization of general monotone properties (as opposed to graph properties) which exhibit coarse thresholds. In this note, we explain the proof of this result with more details.

 Let $(\Om,\pi)$ be a finite probability space and for $n\in\N$, let $(\Om^n,\pi^{\otimes n})$ be the $n$ dimensional product probability space. We will write $x\sim \pprod{n}$ to indicate that $x$ is drawn from $\Om^n$ according to $\pprod{n}.$ Bourgain's result concerns the particular product space $(\zon,\mu_p^{\otimes n})$ where $\mu_p$ is the $p$-biased distribution on $\set{0,1}$. So $\m_p(1)=p$, $\m_p(0)=q:=1-p$. 
We will use the notation $\zon_p$ for $(\zon, \m^{\otimes n}_p)$. 

Throughout, unless otherwise specified, we will write $\Pr[\cdot]$ for $\Pr_{x\sim \pprod{n}}[\cdot]$ and $\Ex[\cdot]$ for $\Ex_{x\sim \pprod{n}}[\cdot]$. If we are in the context of $\zon_p$, then the probability and expectations will be with respect to $\m_p^{\otimes n}$.


In this note, we will consider $f:\Om^n \rightarrow\set{-1,1}$. This will simplify some calculations from Bourgain's proof where the range is taken to be $\set{0,1}$.
We say $f:\zon\rightarrow\set{-1,1}$ is monotone (increasing) if $f(x)\le f(y)$ whenever $x\le y$ component-wise. 
For any subset $S\seq[n]$, we write $x_S$ to refer to the coordinates of $x$ from $S$. In an abuse of notation, sometimes this will refer to a vector of length $\abs{S}$ and sometimes we will want $x_S$ to a be a vector of length $n$. Also, for $S\seq[n]$, we write $\bfo_S$ for the vector of length $n$ with $1$s in the positions corresponding to $S$ and $0$s elsewhere.

 Let $f:\Om^n\rightarrow\set{-1,1}$. The \emph{$i$th expectation operator}, $E_i$, applied to $f$ takes the expectation with respect to variable $x_i$. So
\[E_if(x) = \mathop{\Ex}_{x_i\sim \pi}[f(x_1,\ldots,x_{i-1},x_i,x_{i+1},\ldots,x_n)]\] is a function of $x_1,\ldots,x_{i-1},x_{i+1},\ldots,x_n.$  We also define the \emph{$i$th directional Laplacian operator}, $L_i$, by
\[L_if = f - E_i f.\]  The \emph{influence of coordinate $i$ on $f$} is defined as
\[\infl_i[f] = \angbs{f,L_if} = \angbs{L_if,L_if}\]
where the inner product is defined by
\[\angbs{f,g} = \mathop{\Ex}_{x\sim\pprod{n}}\sqbs{f(x)g(x)}.\]
The \emph{total influence of $f$} is $\tinfl[f] = \sum_{i=1}^n\infl_i[f]$.

Let $f = \sum_{S\subseteq [n]}\fse$ be the generalized Walsh expansion or orthogonal decomposition of $f$.  Recall, that the orthogonal decomposition of $f$ is the unique decomposition that satisfies the following two properties
\begin{enumerate}
\item For every $S\seq [n]$, $\fse(x)=\fse(x_1,\ldots,x_n)$ depends only on $x_i$ for which  $i\in S$.
\item For every $S\seq [n]$, $E_i\fse(x) = 0$ for all $i\in S$.
\end{enumerate}
In the case of $\zon_p$, for any $S\subseteq [n]$, we have that \[\fse(x) = \hf(S)\prod_{i\in S}r(x_i)\] where $r(0) = -\sqrt\frac{p}{1-p}$ and $r(1)=\sqrt\frac{1-p}{p}$.  We also have that $\hf(S) = \mathop{\Ex}_{x\sim \m_p^{\otimes n}}\sqbs{f(x)\prod_{i\in S}r(x_i)}$.

If $S\seq [n]$ and $\bar{S} = [n]\setminus S$, we let $\fss$ represent the function dependent on the coordinates of $S$ where we take the expectation of $f$ over the variables in $\bar{S}$. So if we think of $x$ as $(x_S,x_{\bar{S}})$, then
\[\fss(x) = \fss(x_S) = \mathop{\Ex}_{x_{\bar{S}} \sim \pi^{\otimes\bar{S}}} [f(x_S,x_{\bar{S}})]. \]
$\fes$ and $\fss$ are related by the following two formulas:
\begin{align}\label{fesis}
\fes = \sum_{J\seq S}(-1)^{\abs{S}-\abs{J}}f^{\seq J}
\end{align}
and
\begin{align}\label{fssis}
\fss = \sum_{J\seq S}f^{=J}
\end{align}
Basic Fourier formulas, which hold for the orthogonal decomposition, give us that
\begin{equation}\label{fourierforms1}L_if = \sum_{S\ni i}\fse,\qquad \infl_i[f] = \sum_{S\ni i}\norm{\fse}_2^2 = \sum_{S\ni i}\hf(S)^2\end{equation}
and
\begin{equation}\label{tinflformula}\tinfl[f] = \sum_{i=1}^{n}\norm{L_if}_2^2 = \sum_{S\seq [n]}\abs{S}\norm{\fes}_2^2\end{equation}
where the last equality in \eqref{fourierforms1} holds in the case of $\zon_p$.

For products of general finite probability spaces, we have the following result
\begin{maintheorem}\label{genprodspace}
 For any $f:\Om^n\rightarrow\set{-1,1}$ with $\E{f(x)}=0$ and $\tinfl[f] < C$, we have 
\begin{equation}\label{thm1ineq}
 \E{\mathop{\max}_{0 < \abs{S} \le 10C}\abs{\fss(x)}} > \d
\end{equation}
where $\d = 2^{-O(C^2)}$.
\end{maintheorem}
This result is the main ingredient in Bourgain's proof and it does not rely on the space being $p$-biased bits, so we will prove it here without such an assumption. 

For a monotone boolean function $f:\zon\rightarrow \set{-1,1}$, Margulis \cite{Mar} and Russo \cite{Rus} proved the following relationship between the total influence and the sharpness of the threshold:
\begin{equation}\label{mr}
p(1-p)\frac{d}{dp}\P{f(x)=1}=\tinfl[f]
\end{equation}
where $\Pr$ and $\tinfl$ are both with respect to $\m_p^{\otimes n}$. In other words, the rate of transition of $f$ from $-1$ to $1$ with respect to the rate of increase of $p$ is determined by the total influence. Hence functions with large total influence should have ``sharp'' thresholds and functions with small total influence should have ``coarse'' thresholds.  

Bourgain's result in \cite{F99} is now given. This result basically states the following. Let $f:\zon\ra\set{-1,1}$ be a monotone boolean function and let $p$ be the critical probability (which is allowed to approach 0 rapidly with $n$), when $f$ is equally likely to be $-1$ or $1$. Then if $f$'s total influence is bounded, either (1) a non-negligible portion (according to $\mu_p^{\otimes n}$) of the $x$'s with $f(x)=1$ have a small witness, or (2) there exists a small set of coordinates such that conditioning on these coordinates being 1 boosts the expected value of $f$ by a non-negligible amount.  Keep in mind that in the following statement, $\Ex,\Pr$ and $\tinfl$ are with respect to $\m_p^{\otimes n}$.

\begin{maincoro}\label{mainpbiased}
Let $f:\set{0,1}^n \ra \set{-1,1}$ be monotone (increasing) and suppose that $p=p(n)$ is such that $\E{f}=0$ and $\tinfl[f] < C$. Then there exists some $\d' = 2^{-O(C^2)}$ such that if $p < \frac{\d'}{20C}$ then
at least one of the following two possibilities holds:
\begin{enumerate}
\item  
\begin{equation}\label{condition1}\P{\exists S \subset [n],\,  \abs{S} \le 10C,\,\bfo_{S}\le x,\,f(\bfo_S)=1} > \d'.\end{equation}
\item There exists $S'\seq[n]$ with $\abs{S'} \le 10C$ with $f(\bfo_{S'})=0$ such that 
\begin{equation}\label{condition2}
f^{\seq S'}(\bfo_{S'}) > \d'.
\end{equation}
\end{enumerate}
\end{maincoro}

\begin{proof}[Proof of Corollary \ref{mainpbiased} ]
Let $\d' = \d/2$ where $\d$ is given by Theorem \ref{genprodspace}. Suppose that the first alternative of the theorem, \eqref{condition1}, does not hold, i.e., 
\begin{align} \label{notcond1}
	\P{\exists S \subset [n],\,  \abs{S} \le 10C,\,\bfo_{S}\le x,\,f(\bfo_S)=1}  \le \d'.
\end{align}
Then applying Theorem \ref{genprodspace}, if $n$ is sufficiently large, there must exist $\bar{x} \in \zon$ and $S\seq[n]$, $\abs{S}\le 10C$ such that
for all $x' \le \bar{x}$ with at most $10C$ $1$'s, we have $f(x')=0$
and
\begin{align}\label{deltaover2}\abs{\fss(\bar{x})} > \d'.\end{align}
Now by monotonicity of $f$, we have for all $S, x_S$
\begin{align*}
\fss(x_S) &= \Ex_{x_{\bar{S}} \sim \m_p^{\bar{S}}} [f(x_S,x_{\bar{S}})] \\
&\ge \Ex_{x_{\bar{S}} \sim \m_p^{\bar{S}}} [f(\vec{0}_S,x_{\bar{S}})] \\
&\ge \E{f(x) - \sum_{i\in S}\bfo_{\set{x_i = 1}}} \\
&= - p\abs{S} > \frac{\d'}{2}.
\end{align*}
So \eqref{deltaover2} implies that \[\fss(\bar{x}) >\d'\]
which implies the second alternative of the theorem, \eqref{condition2}, by taking
\[S' = S \cap \set{i : \bar{x}_i = 1}.\]

\end{proof}

The following easy corollary may be a useful statement.

\begin{maincoro}
Let $f:\set{0,1}^n \ra \set{-1,1}$ be monotone (increasing) and suppose that $p=p(n) < \frac{\d}{100C}$ is such that $\E{f}=0$ where $\d = 2^{-O(C^2)}$. Furthermore, suppose that $\tinfl[f] < C.$ Then there exists a subset $S\seq[n]$ with $\abs{S}\le 10C$ such that
\begin{equation}\label{coreqn}\E{f(x)\mid x_S=(1,\ldots,1)} > \d.\end{equation}
\end{maincoro}
To derive this from Corollary \ref{mainpbiased}, note that if the first alternative holds, then there exists a small $S$ which makes the expectation in \eqref{coreqn} equal to 1. If the second alternative holds, note that \eqref{condition2} and \eqref{coreqn} are equivalent.

As a corollary of his very general theorem, Hatami \cite{Hat} proves that in fact the expectation in \eqref{coreqn} can be made arbitrarily close to 1. The size of the guaranteed $S$ may have size exponential in $C^2$, but it is still independent of $n$. 

\section{The Proof}

\begin{proof}[Proof of Theorem \ref{genprodspace}]
First observe that the facts $\norm{f^{=\es}}_2^2 = 0$ and $\sum_{S\seq[n]}\norm{\fes}_2^2 = 1$ and the assumption that $\sum_{S\seq[n]}\abs{S}\norm{\fes}_2^2 < C$ imply that
\begin{equation} \label{onetenthlb}
\frac{9}{10} \le \sum_{0 < \abs{S} \le 10C}\norm{\fes}_2^2
\end{equation}
since  
\begin{align*}
\sum_{\abs{S} > 10C} \norm{\fes}_2^2 &\le \sum_{S\seq[n]}\frac{\abs{S}}{10C}\norm{\fes}_2^2 \\
&<\frac{C}{10C} = 1/10.
\end{align*}

Now, consider the following functions
\[h_i(x):=\of{\sum_{\stackrel{S\ni i}{ \abs{S}\le 10C}}\abs{\fse(x)}^2}^{1/2}\]
and 
\[h(x) = \of{\sum_{\abs{S}\le 10C}\abs{\fes(x)}^2}^{1/2}\]
By Prop. 6 of \cite{B79}, we may say that for a fixed $1<q\le 2$, we get
\begin{align}\norm{h_i(x)}_q^q &\le c_1\norm{L_i f}_q^q \nonumber \\
&= c_1\E{\abs{L_if}^q} \nonumber\\
&\le c_2\E{\abs{L_if}} \nonumber\\
&\le C_1\E{(L_if)^2} \label{2.4} \nonumber\\
&= C_1\infl_i[f].
\end{align}
with $C_1=C_1(q) =2^{O(C)}$ and $c_1,c_2$ are some constants which also depend only on $q$.  The reader should note that in the proof that follows, we will only apply the result of \cite{B79} with $q=4/3.$ If $q'=\frac{q}{q-1}$, then we also have
\begin{equation}\label{qprimeineq}\norm{h(x)}_{q'} \le C_1\norm{f}_{q'} = C_1.\end{equation}
Hence we have
\begin{align}\label{sumqnorms}\sum_{i=1}^n\norm{h_i(x)}_q^q \le C_1\sum_{i=1}^n\infl_i[f] \le C\cdot C_1.\end{align}

Let $0<\eps < M < \infty$ be constants which are taken to be $\eps = 2^{-O(C)}$ and $M=O\bfrac{1}{\eps}$ and let
\begin{align}
\eta_i(x) &= \bfo_{\{h_i(x) > \eps\}}\\
\xi(x) &= \bfo_{\{\sum_i\eta_i(x) < M\}}.
\end{align}
Specific values for $M$ and $\eps$ may be determined in terms of $C$ and $C_1$ by analyzing the inequalities that follow.

Now $1-\xi(x)$ is the indicator of the event that there are more than $M$ coordinates $i$, such that $h_i(x) >\eps.$ Given relation \eqref{tinflformula} and the assumption that total influence is bounded, we should expect this event to have small probability.  Hence we have, using Markov's theorem twice, that
\begin{align*}
\E{1-\xi(x)} &\le \frac{1}{M}\E{\sumin\eta_i(x)} \\
&\le \frac{1}{M\eps^2}\E{\sumin h_i(x)^2} \\
&\le\frac{1}{M\eps^2}\E{\sumin \sum_{S\ni i} \fes(x)^2 }\\
&\le\frac{C}{M\eps^2}.
\end{align*}

Now, inequality \eqref{onetenthlb} tells us that
\[\frac{9}{10} < \E{\sum_{0<\abs{S}\le 10C}\fes(x)^2}.\]
Note that for any $x$, either there exist $\ge M$ many $i$ such that $h_i(x) > \eps$, or there are $<M$ such $i$. In the latter case, there are two types of $S\seq[n]$ with $0<\abs{S}\le 10C$: those which contain an $i$ such that $h_i(x)\le \eps$ and those containing only $i$'s such that $h_i(x)>\eps$.

Hence, using the indicator functions $\xi, \eta_i$, and recalling the definitions of $h(x)$ and $h_i(x)$, we may split up the following expectation as
\begin{align}
\E{\sum_{0<\abs{S}\le 10C}\fes(x)^2}  &\le \E{h(x)^2(1-\xi(x))}\label{term1}\\
&+ \E{\sumin h_i(x)^2(1-\eta_i(x))}\label{term2} \\
&+ \E{\sumtenc \fes(x)^2\of{\prod_{i\in S}\eta_i(x)}\xi(x)}.\label{term3} 
\end{align}
We now bound each of these terms in turn.

For \eqref{term1}, we apply Cauchy-Schwarz and see that
\begin{align}
\E{h(x)^2(1-\xi(x))} &\le \E{h(x)^4}^{1/2}\E{(1-\xi(x))^2}^{1/2}\label{1.1} \\
&\le \norm{h(x)}_4^2\E{1-\xi(x)}^{1/2}\label{1.2}\\
&\le C_1^2\cdot\sqrt{\frac{C}{M\eps^2}}\label{1.3}
\end{align}
where we used \eqref{qprimeineq} with $q'=4$ (and hence $q=4/3$) to go from \eqref{1.2} to \eqref{1.3}.

For \eqref{term2}, we note that in this expectation, $h_i(x) \le \eps$ for any $x$ such that $\eta_i(x)=0$. Also, since each $h_i$ is a positive function, we may write $h_i^2 = h_i^{2/3}h_i^{4/3}$. So
\begin{align*}
\E{\sumin h_i(x)^2(1-\eta_i(x))} & = \sumin \E{h_i(x)^{2/3}h_i(x)^{4/3}(1-\eta_i(x))} \\
&\le \eps^{2/3}\sumin \E{\abs{h_i(x)}^{4/3}} \\
&=\eps^{2/3}\sumin \norm{h_i(x)}_{4/3}^{4/3}\\
&\le \eps^{2/3}\cdot C\cdot C_1
\end{align*}
where we used \eqref{sumqnorms} with $q=4/3$ to get the last line.

Finally, for \eqref{term3}, we first observe that for any $x$, we have that
\[\sumtenc\of{\prod_{i\in S}\eta_i(x)}\xi(x) < M^{10C}\]
since if $\xi(x)=1$, then $\scr{M}_x =\set{i : \eta_i(x) = 1}$ has $\abs{\scr{M}_x} < M$. So the non-zero terms in the sum correspond to $S\subseteq \scr{M}_x$, $0 < \abs{S} \le 10C$. So we get
\begin{align*}
&\E{\sumtenc \fes(x)^2\of{\prod_{i\in S}\eta_i(x)}\xi(x)} \\
 &\le \E{\max_{0<\abs{S} \le 10C}\fes(x)^2 \sumtenc\of{\prod_{i\in S}\eta_i(x)}\xi(x)} \\
&\le M^{10C}\E{\max_{0<\abs{S} \le 10C}\fes(x)^2}.
\end{align*}

Adding these three estimates gives
\begin{align}
\frac{9}{10} < C_1^2\sqrt{\frac{C}{M\eps^2}} + \eps^{2/3} C C_1 +  M^{10C}\E{\max_{0<\abs{S} \le 10C}\fes(x)^2}.
\end{align}
Now, by taking $\eps = 2^{-O(C)}$ and $M = O(1/\eps)$, we easily have that
\[\E{\max_{0<\abs{S} \le 10C}\fes(x)^2} > 2^{-O(C^2)}.\]
Now note that for any $S\seq [n]$, by \eqref{fesis},
\begin{align}
 \abs{\fes} &= \abs{\sum_{J\seq S}(-1)^{\abs{S}-\abs{J}}f^{\seq J}} \nonumber \\
 &\le \sum_{J\seq S}\abs{f^{\seq J}} \le 2^{\abs{S}}\max_{J\seq S}\set{\abs{f^{\seq J}}}\label{other bound} \\
 & \le 2^{\abs{S}} \label{other bound2}.
\end{align}
since $\abs{f^{\seq S}} \le 1$.
So applying \eqref{other bound} and \eqref{other bound2} and using the fact that $f^{\seq\es} =\E{f}= 0$, we have
\begin{align}
 2^{-O(C^2)} &< \E{\max_{0<\abs{S} \le 10C}\fes(x)^2} \nonumber\\
&= \E{\max_{0<\abs{S} \le 10C}\abs{\fes(x)}\abs{\fes(x)}} \nonumber \\
&\le 2^{10C}\E{\max_{0<\abs{S} \le 10C}\abs{\fes(x)}} \nonumber\\
&\le 2^{20C}\E{\max_{0<\abs{S} \le 10C}\max_{J\seq S}\abs{f^{\seq J}(x)}} \nonumber \\
& = 2^{20C}\E{\max_{0<\abs{S} \le 10C}\abs{\fss(x)}} \nonumber
\end{align}
which completes the proof.


\end{proof}

\end{document}